\documentclass{amsart}
\usepackage{graphicx}

\usepackage{amsthm,amsmath,verbatim,amssymb}
\usepackage{color}
\newtheorem{lem}{Lemma}

\newtheorem{prop}{Proposition}
\newtheorem{thm}{Theorem}
\newtheorem{cor}{Corollary}

\newcommand{\CP}{\mathbb{CP}}
\newcommand{\kahler}{K\"ahler }
\renewcommand{\d}{\partial}
\newcommand{\dbar}{\bar\partial}
\newcommand{\om}{\omega}
\newcommand{\dcal}{\mathcal D}

\newcommand{\ddbar}{\partial\dbar}
\newenvironment{example}{\medskip\noindent{\it Example:\/} }{\medskip}
\newcommand{\R}{\mathbb{R}}

\newtheorem{rem}{Remark}

\newcommand{\ecal}{\mathcal{E}}

 \def   \half   {{\frac{1}{2}}}
\newcommand{\C}{\mathbb{C}}
\def \to {\rightarrow}
\title{Energies of Random zeros on }
\begin{document}\title []
{Random Riesz energies on compact K\"{a}hler manifolds}
\author[Renjie Feng]{Renjie Feng}
\author[Steve Zelditch]{Steve Zelditch}
\address{Department of Mathematics, Northwestern  University, Evanston, IL 60208, USA}
\email{renjie@math.northwestern.edu \\ zelditch@math.northwestern.edu}

\thanks{Research partially supported by NSF grant  \# DMS-0904252.}

\date{\today}

\begin{abstract}The expected   Riesz energies $E_{\mu^N_h}\mathcal E_{s}$
of the zero sets of   systems of independent  Gaussian random polynomials are determined asymptotically
as the degree $N \to \infty$ in all dimensions and codimensions.   The asymptotics are
proved for sections of any positive line bundle over any compact K\"ahler manifold.
\end{abstract}

\maketitle
\section{Introduction}

A classical problem is to determine the optimal  configurations of $N$  points $\{z^1, \dots, z^N\}$  on the unit sphere $S^2$,  i.e. configurations which  minimize some  energy  function,
 \begin{equation}\label{def}
\mathcal E_{s}(z^1,\cdots,z^N)=\sum_{i\neq j} G_s(z^i,z^j).
 \end{equation}
Natural choices of the kernel $G_s(z,w)$ are the Riesz kernels,
\begin{equation}\label{riea} G_s(z,w)= \left\{ \begin{array}{ll} r_g(z,w)^{-s}, & \,\,\,  0 < s<2m \\ & \\
 - \log r_g(z,w), &\,\,\,s=0 . \end{array} \right. \end{equation}
Here,  $r_g(z,w) $ is the geodesic distance between $z, w$. The assumption $0<s<2m$ ensures that the Riesz function $r_g(z,w)^{-s}$ is integrable.
The minimal  energy is  denoted by
\begin{equation} \label{MINEN} \ecal_{s}( N) =  \min_{\{z_1, \dots, z_N \}} \mathcal E_s (z^1,\dots,z^N). \end{equation}
 For background
 on this problem, we refer to \cite{BBGKS, HS}.  The problem  of  minimizing energy is of interest  on any compact Riemannian manifold, although it seems mainly to have been studied on the standard sphere.

 Of course, it is a very demanding problem
to determine the actual optimizers. Related problems are to find the asymtotics
of the minimial energy \eqref{MINEN} as $N \to \infty$ and to determine the properties of asymptotically minimizing
configurations, which are sometimes called elliptic Fekete points.  For instance, to minimize energy  the points must be well-separated and equidistributed with respect to the volume
form.

This article is concerned with the expected energies $E_{\mu^N_h}\mathcal E_{s}$  of
the zero sets of Gaussian random systems of holomorphic polynomials and their generalizations to
any compact \kahler manifold. Thus, we do not study   optimal configurations but rather
consider the
separation and equidistribution properties of  configurations defined by random zero sets as measured by the energies \eqref{def}.

Our results generalize those of a  recent pair of articles \cite{Z,ABS} in the complex one
dimensional case, which showed that the zeros of Gaussian random holomorphic polynomials
of one complex variable are asymptotic energy minimizing configurations to a certain order in $N$.
The results are reviewed in \S \ref{Zh}.  Such polynomials are  holomorphic sections of ${\mathcal O} (N) \to \CP^1 = S^2$,
the $N$th power of the hyperplane bundle over the Riemann sphere. It has been known for some time that  in complex dimension one, the zeros of random
polynomials of degree $N$  (or of any holomorphic sections) repel each other so that the nearest neighbors are almost always
of order  $\frac{1}{\sqrt{N}}$ apart. More precisely, the  probability of pairs being $\frac{C}{\sqrt{N}}$
apart can be determined   from the scaling asymptotics of the two-point correlation function of
zeros in \cite{H, BSZ1,BSZ2,BSZ3}.  Furthermore,  the `empirical measure of zeros' defined by
\begin{equation}\label{EMP}  \mu_z = \frac{1}{N} \sum_{z: s^N(z)=0} \delta_z \end{equation}
tends to the `equilibrium measure' given by the Riemannian area form \cite{SZ2}.
The leading order asymptotics of the energy mainly reflect this repulsion and equidistribution.
Repulsion  and equidistribution are not sufficient to ensure that random zero point configurations asymptotically minimize energy, and indeed one would
not expect a random configuration to do as well as a carefully chosen one.  But in
complex dimension one, random zero sets do surprisingly well (see \S \ref{Zh}).

 In this article, we consider the expected energies of the zero ses
of a random  full system of $m$  holomorphic polynomials  on a complex projective space $\CP^m$ of any
dimension. In fact, we consider the problem for sections of ample holomorphic line bundles $L \to M$  over general \kahler
manifolds, and moreover for systems of $k \leq m$ holomorphic sections.
We refer to the case of a full
system of $m$ polynomials in complex dimension $m$, whose simultaneous zero set is (almost surely) a discrete
set of points, as the {\it point case} or equivalently as the {\it discrete case}.

In higher dimensions, we do not expect the complex zeros in the discrete case to
asymptotically minimize energy.  They do become equidstributed with respect to
the equilibrium measure but  they are not well separated.
As recalled in \S \ref{PCC},    the scaling limits of the pair correlation function of
zeros were calcuated in \cite{BSZ1,BSZ2,BSZ3}  in all dimensions and co-dimensions. The results show that, in the point case,  the zeros almost ignore each other
with only a slight repulsion at very close distances in dimension 2, and they in fact attract or clump together in
higher dimensions. The results of this article  give a somewhat different perspective on the
separation of the zeros as reflected by energy asymptotics, or more precisely on the distance of zeros
from the discriminant variety of polynomial systems with a double root (see \S \ref{CONFIG}).

Before stating our results on expected energies, let us review what is known about the minimum energies
of configurations of points in the point case so that we can compare expected energies
to minimal energies.

\subsection{Minimal Energies}\

First, we review results on  $\CP^1 = S^2$.  We caution that
different authors use slightly different normalizations of the radius of the sphere and
the choice of Riesz kernels \eqref{riea} and  the coefficients of the asymptotic
energies reflect these choices.

In  \cite{KS,BBP,RSZ}, the authors consider the
energy \eqref{def} with
$$G_s(z,w) = \frac1{[z,w]^s},\,\,\, (s\neq 0); \; \mbox{resp.}  \;\; G_0(z,w)  =   \log
\frac{1}{[z, w ]},\,\,\,\, (s=0). $$ Here, $[z,w]$ is the chordal distance
between two points on $S^2 \subset \R^3$ rather than  the
geodesic distance as in (\ref{riea}).  The are related by $[z,w] = 2 \sin \frac{r}{2}$ on
the sphere of radius $\half$ (i.e. area $\pi$).
  In   \cite{KS, W1, W2},  it is proved that

\begin{equation}\label{6} \left\{ \begin{array}{ll} \ecal_{2}( N) \sim \frac{1}{8}  N^2 \log N ,  & s = 2 \\ &  \\ C_1(s) N^{1+\frac s2}\leq \ecal_{s}( N) \leq
C_2(s) N^{1+\frac s2}, & s > 2, \\ & \\
\frac12V_2(s)N^2-C_3(s)N^{1+\frac s2}\leq \ecal_{s}( N) \leq \frac12V_2(s)N^2-C_4(s)N^{1+\frac s2}, & s < 2 \end{array} \right. \end{equation}
where all $C_j(s)$ and
$V_2(s)$ are positive constants independent of $N$.

The  asymptotic energy $\ecal_{0} (N)$ on the standard $S^2$ was further  studied in \cite{BBP, RSZ}. Theorem 3.1 of \cite{RSZ} asserts that
if $C_N$ is defined by
\begin{equation}\label{dhdhd}
\ecal_{0}( N) =-\frac{1}4\log (\frac{4}{e})N^2-\frac{N}{4}\log N+C_N N,
\end{equation}
then for the standard $S^2$ one has
$$-\frac1 4\log \frac{\pi \sqrt 3}{2}-\frac{\pi}{8\sqrt 3}\geq\limsup_{N \to \infty} C_N\geq \liminf_{N \to \infty} C_N \geq - \frac{1}{4} \log\left( \frac{\pi}{2}(1 - e^{-a})^b \right), $$
for certain explicit constants $a, b$. But the precise value of $C_N$ is not  known.

In higher dimensional sphere $S^m$, results of \cite{B, KS} show that

\begin{equation}\label{minm2} \left\{ \begin{array}{ll}
C_1(s)N^{1+s/m}\leq \ecal_s(N)\leq C_2(s)N^{1+s/m}, & s > m \\ & \\
 \ecal_m(N)\sim \gamma_m N^2\log N, & s = m \\&\\
 \frac1
2 V_m(s)N^
2- C_3(s) N^{
1+\frac sm}\leq \ecal_s(N)\leq \frac1
2 V_m(s)N^
2- C_4(s) N^{
1+\frac sm}, & s<m
\end{array} \right. \end{equation}
where $V_m(s)$,  $\gamma_m$ and $C_j(s)$ are positive constants independent of $N$.
Many further references and results can be found in \cite{BBGKS}.

In the case of the Green's function $G(z,w)$ of any compact Riemannian manifold, i.e. the kernel of $\Delta^{-1}$ on the orthogonal complement of
the constant functions,   N. Elkies proved the general lower bound
$$\ecal_{\mbox{Green}}(N)  \geq  -\frac{1}{4} \log (\frac{4}{e}) N^2 - \frac{1}{4 } N \log N
- \frac{11}{6 \pi} N + o(N). $$  Proofs can be found in  \cite{Hr, La,HiS}.  In complex dimension one, the lower bound follows
from Mahler's inequalities  \cite{M}. Although the right side is very precise, we are not aware of a comparable upper bound
in the literature.

\subsection{\label{Zh} Expected  Energies of zeros of random polynomial systems}\

We now consider the energies of random  configuration of points which are
defined as the  zeros of systems of  random polynomials (or sections). Let  $(L,h) \to (M, \omega)$ be  a positive holomorphic line bundle equipped with a Hermitian metric $h=e^{-\phi}$  ($\phi$ is the \kahler potential.   Let $H^0(M,L^N)$ be the space of the global holomorphic sections of the line bundle $L^N$; put $d_N:=\dim H^0(M,L^N)$. For notation and background we refer to \cite{GH,SZ2}.  We normalize the volume form by  \begin{equation} \label{TILDEV}   d\tilde V:=\left(\frac{\omega(z)}{\pi}\right)^m,\;\;\;
\int_M \left(\frac {\omega}{\pi}\right)^m=1. \end{equation}

We denote the zero set of a full system $\{s_1^N, \dots, s_m^N\}$ of $m$ polynomials (or sections) in dimension $m$ by
\begin{equation} \label{ZEROSMAP} Z_{{s^N_1,\cdots,s^N_m}} = \{(z^1, \dots, z^m) \in M: s^N_j(z^1, \dots, z^m) = 0, \; \forall j = 1, \dots, m\}. \end{equation}
The Riesz energy \eqref{def} of a configuration of zeros is thus given by
 \begin{equation}\label{randomly}\mathcal E_{s}(Z_{{s^N_1,\cdots,s^N_m}}) =\sum_{\substack{z^i\neq z^j\\z^i,z^j\in Z_{{s^N_1,\cdots,s^N_m}}}}G_{s}(z^i,z^j)\end{equation}

As in \cite{SZ1,SZ2,BSZ1} and elsewhere, we endow the spaces  $H^0(M, L^N)$ with  Gaussian probability measures. They
are induced by the inner product defined as follows: Let
 $h^N$ on $L^N$ denote the $N$th power of the Hermitian metric  $h$ on $L$, and define
\begin{equation}\label{innera}\langle s_1,  s_2 \rangle = \int_M h^N(s_1,
s_2)\,\frac 1{m!}\om^m\;,\qquad s_1, s_2 \in
H^0(M,L^N)\,\;.\end{equation}  The \ {\em Hermitian Gaussian measure}
on
$H^0(M,L^N)$ is the complex Gaussian probability measure $\mu_h^N$
induced on $H^0(M, L^N)$ by the inner product \eqref{innera}:
\begin{equation} \label{HGM} d\mu_h^N(s)=\frac{1}{\pi^m}e^
{-|c|^2}dc\,,\qquad s=\sum_{j=1}^{d_N}c_jS^N_j\,,\end{equation}
where $\{S_1^N,\dots,S_{d_N}^N\}$ is an orthonormal basis for
$H^0(M,L^N)$.

The Gaussian measure $d\mu_h^N$ induces product measures on $(H^0(M, L^N))^k = H^0(M, L^N) \times \cdots
\times H^0(M, L^N)$ ($k$ times). For simplicity of notation we continue to denote them by $d\mu_h^N$, leaving
it to the context to make clear how may factors of $H^0(M, L^N)$ are involved. Product measure corresponds
to choosing  $k \leq m$ independent random sections, hence a Gaussian random system.
We define $E_{\mu^N_h}\mathcal E_{s}$ to be the expected (average) value of
the energy of the zeros of Gaussian random sections chosen from the
ensemble $(H^0(M,L^N),\mu_h^N)$. In the discrete case,
\begin{equation}\label{defd}
E_{\mu^N_h}\mathcal E_{s}=\int_{(H^0(M,L^N))^m}\mathcal E_{s}(Z_{s^N_1,\cdots,s^N_m})d\mu_h^N(s)=\int_{(\C^{d_N})^m}\mathcal E_{s}(Z_{s^N_1,\cdots,s^N_m})\frac{e^{-|c|^2}}{\pi^{d_N}}dc,
\end{equation}

Before stating our results in the discrete case, we recall the results of \cite{Z,ABS} when $m = 1$.
In the special case of the standard metric on $S^2$ and with the logarithm Riesz energy, it is proved in  \cite{Z, ABS} that \begin{equation}  \label{sphere}
E_{\mu_{h^N_{FS}}}\mathcal E_{{0}}= \left\{  \begin{array}{cc} \frac{N^2}{4}-\frac1{4}N\log
N-\frac{N}{4}+o(N), & \mbox{radius}\; = \half \\ &\\
- \frac{1}4\log(\frac 4 e)N^2-\frac1{4}N\log
N +   \frac{1}{4} \log(\frac 4 e) N +o(N), & \mbox{radius} \; = 1.
\end{array} \right. \end{equation}
In  our notation, $S^2$
is the \kahler manifold  $(\CP^1,\omega_{FS})$,
where $\omega_{FS}$ (or $h_{FS}$)  is the Fubini-Study metric and the holomorphic line bundle $L$ is hyperplane line bundle $\mathcal O(1)$; the area of
 $(\CP^1, \omega_{FS})$ equals $1$, hence   is the round $S^2$  of radius
$\frac1
2$.
The expected energy of $N$ randomly chosen points (i.e. from the binomial point process)
equals $- \frac{1}4\log(\frac 4 e)N^2 +   \frac{1}{4} \log(\frac 4 e) N .$
 Thus, the  $- N \log N$ term  reflects the lower energy of well-separated
points.

To be precise,   Q. Zhong  \cite{Z} obtained general asymptotic results  for any \kahler metric on any Riemann surface,
where the energy was defined in terms of the Green's function,
 \begin{equation}\label{green}G_0(z,w)=-\frac 1 2\log r_g(z,w)+F_g(z,w)\end{equation}
 where $F_g$ is the Robin function such that $\int_M G_0(z,w)\omega(z)=0$.
 In the case of the Green's energy,
 Zhong proved that
 \begin{equation}\label{main theorem}E_{\mu_h^N} \mathcal E^N_{0}= -\frac1{4}N\log
N-\frac{N}{4}-N\int_M F_{g}(z,z)\omega+o(N).
\end{equation}
where $\int_M F_{g}(z,z){\omega}$ is the Robin constant.  The $N^2$ term vanishes since the integration of the Green's function is $0$.
In the special case   of
the standard metric on $S^2$,
 Armentano-Beltran-Shub \cite{ABS} later showed that the expected energy is exactly
given by the formulae in \eqref{sphere} without the $o(N)$ term.

   The asymptotic results of \cite{Z, ABS} show that the expected energy of zero point configurations
is asymptotically the same as the minimum energy of any configuration up to order $N$ (it appears
to remain unknown what the coefficient of the $N$ term is for the minimal energy).
See also \cite{ZZe} for a correction and addendum to the calculations in \cite{Z}.

\subsection{Main Results in the discrete case}\

In this article, we consider random Riesz s-energies for zeros of polynomial systems
and their analogues on  a general
compact \kahler manifold of any dimension $m$  and for a certain range of $s$. In the discrete case we obtain the following
asymptotic expansion for the average Riesz energy.

\begin{thm}\label{dis} Let $(M,\omega)$ be an $m$-dimensional  compact \kahler manifold and let $(L, h) \to (M, \omega)$
be a positive Hermitian holomorphic line bundle.   Then the expected
 Riesz energies defined in (\ref{defd}) admit the following asymptotics,
\begin{itemize}
\item If $0<s< \min\{2m, 4\}$,
\begin{align*}E_{\mu_h^N}\mathcal E_{s}=&a_1(h,s)N^{2m}+\sum_{j=2}^{p-1}a_j(h,s)N^{2m-j}\\&+c_m(s)N^{m+\frac{s}{2}}+O(N^{m+\frac{s-1}{2}}(\log N)^{m-\frac{s}{2}})\end{align*} where $p =[m-\frac{s}{2}]+1$, where $a_j(h, s)$ are the coefficients in Lemma \ref{k1},  and
 where
$$c_m(s) = 2m\int_{0}^{\infty} \left[\kappa_{mm}(r)-1\right]r^{2m-1-s}dr . $$
 \item In the case of  logarithmic energy  when $s=0$, we have,\begin{align*}E_{\mu_h^N}\mathcal E_{0}&=a_1(h,0)N^{2m}+a_2(h,0)N^{2m-2}+\cdots+[a_m(h,0)-c_m]N^m\\& -N^{m }\log \sqrt N
+O(N^{m-\frac{1}{2}}(\log N)^{m+1}),\end{align*}
where
 \begin{equation}\label{eq:dddrreneweq}c_m=2m\int_0^{\infty} \log r\left[\kappa_{mm}(r)-1\right]r^{2m-1}dr.\end{equation}

 \end{itemize}

 \end{thm}

  Above, the  coefficient $c_m(s)$ involves the
scaling limit  $\kappa_{mm}(r)$ of the pair
correlation function of zeros in the discrete case in dimension $m$. As  reviewed in \S \ref{PCC},   it has the small distance asymptotics,
\begin{equation} \label{leadinga} \kappa_{mm}(r)= \frac{m+1}{4}
r^{4-2m} + O_m(r^{8-2m})\,,\qquad\mbox{as }\ r\to 0\,.\end{equation}
Hence the integral defining the coefficient  $c_m(s)$  only converges if $s \leq 4$.
It is independent of the Hermitian metric $h$.

The coefficients $a_j(h,s)$ are integrals of $G_s(z,w)$ against  curvature invariants  of $h$ and  can be explicitly computed by the algorithm in Lemma \ref{k1}. In particular,  the first two are given by
\begin{equation}\label{coe}
\begin{aligned}
a_1(h,s)=&\int_{M\times M}G_s(z,w)d\tilde V(z)d\tilde V(w)\\
a_2(h,s)=&-2\int_{M\times M}G_s(z,w) \Delta_\omega S(z)d\tilde V(z)d\tilde V(w)
\end{aligned}\end{equation} where $\Delta_\omega$ is the complex Laplace operator with respect to the \kahler form $\omega$, and $S$ is the scalar curvature of $\omega$.

We can give a more explicit evaluation in
  the special case of $(\CP^m, \omega_{FS})$  and  $(L , h)= (\mathcal O(1), h_{FS})$. As is well known,
in a suitable frame, we  can identify $H^0(\CP^m,\mathcal O (N))$ with the space of polynomials of $m$ variables of degree $N$
\cite{GH}.
\begin{cor} \label{1}  For the space $H^0(\CP^m,\mathcal O (N))$ of polynomials of degree $N$  equipped with the Gaussian
Hermitian measure  induced by $h_{FS}$, the Riesz energies
of zeros of $m$ independent random polynomials are given by

\begin{itemize}
\item If $0<s< \min\{2m, 4\}$,
$$E_{\mu_h^N}\mathcal E_{s}=  a_1(h_{FS},s)N^{2m} + c_m(s)N^{m+\frac{s}{2}}+O(N^{m+\frac{s-1}{2}}(\log N)^{m-\frac{s}{2}})$$

 \item For the case of logarithm random energies, we have,$$E_{\mu_h^N}\mathcal E_{0}= a_1(h_{FS},0)N^{2m} - N^m\log \sqrt N-c_m N^m+O(N^{m-\frac{1}{2}}(\log N)^{m+1}).$$
 \end{itemize}
 \end{cor}
Thus, all of the terms $a_j(h_{FS}, s)$ vanish for $j\geq 2$.  The proof follows directly from the proof of Lemma \ref{k1} and  based on the fact that the Fubini-Study
metric is balanced, i.e. the Szeg\"o kernel is constant on the diagonal.

The  asymptotics reflect the local behavior (repulsion/attraction) of the zero-point configurations
and the long-range order (uniform distribution) as follows:
\begin{itemize}

\item
The leading coefficient $a_1(h, s)$ is the ``Poisson term", i.e. the scaling limit
of the expected energy of a binomial point process where one chooses $N^m$ points at random from $M$, i.e. a configuration at random from  $M^{(N^m)}  =
\underbrace{M\times\cdots\times M}_{N^m}/S_{N^m}$ (the $N^m$-fold symmetric product)
with respect to $d \tilde{V}$. See \S \ref{CONFIG} for further discussion.

\item The  coefficients $a_j(h,s)$ are global  lower order corrections due to the non-constant curvature of the
K\"ahler metric.

\item The terms $c_m(s)N^{m+\frac{s}{2}}$ and $-N^m\log \sqrt N-c_mN^m$ reflect the local behavior (i.e. separation,
repulsion or clumping)
of the zeros.  They are derived from the integration of Riesz functions against the current of two point correlations near the diagonal of $M\times M$. The separation of the
zeros depends on the dimension $m$.

\item   The overall asymptotics reflects a competition between the global and local behavior of zeros.   The larger
$s$ is, the more significant is the separation of zeros.
When $s>2$, the local behavior of zero-point configurations dominates the global behavior; when $s<2$, the global behavior dominates the local one.

\end{itemize}

The asymptotic average energy of zero point configurations may be compared with the
minimum energy \eqref{minm2}, keeping in mind that the expected number of points in
the configuration of zeros is  $N^m$. Of course, $S^n$ is not a K\"ahler manifold for
$n \geq 3$, so no exact comparison can be made between the minimum and average
energies. The closest K\"ahler analogue of a round sphere $S^{2m}$  is $\CP^m$ with its Fubini-Study
metric of constant positive curvature. For Riesz energies with  $0 < s < \min\{2m, 4\}$  the results of Corollary \ref{1} may be compared
with the result for the standard $S^{2m}$  that $$
\frac1
2 V_{2m}(s)N^
{2m}- C_3(s) N^{
m+s/2} \leq \ecal_s(N^m)\leq
\frac1
2 V_{2m}(s)N^{2m}- C_4(s) N^{m
+s/2}. $$
The exponent $s/2$ arises because we have $N^m$ points in real dimension $2m$.
The leading coefficients of the minimum and average energies are both given by the Poissonian term and
may be considered as the same.  Hence the difference between minimum energy and
mean energy of zero point configurations is first detected in the subleading terms of order $N^{m + \frac{s}{2}}$,
which as discussed above is due to the separation between points.
The coefficient $C_j(s)$ is not known explicitly, but it is significant to  compare the signs of $C_j(s) $ and $c(s)$. It is
known that $C_j(s) > 0$, so that the sub-leading term of the minimum energy is negative. Intuitively, we expect
that $c_m(s) > 0 $ for  $s$ near $4$  when $m \geq 3$ due to the clumping of the zeros as compared to the repulsion
that must take place in an optimal configuration.
The exact formula for $\kappa_{mm}(r)$ is given in \cite{BSZ3} and recalled
in \eqref{pointpaircor}. It  is very complicated and represents the pair correlation function
for zeros of random entire functions in the Bargmann-Fock model, i.e. entire analytic functions on
$\C$ which are in $L^2(\C, e^{- |z|^2}). $ $\kappa_{mm}(r)$ blows up to $+ \infty$ for $r \to 0$ if $m \geq 3$,  supporting the intuition that
$c_m(s)$ should be positive. But it also has negative values in a certain interval,
which represents  a kind of
repulsion relative to the Poisson spatial process for intermediate values of $r$, and that complicates
the discussion.

However we can justify the intuition for $s$ near $4$:

\begin{prop}  \label{POS} For any $m \geq 3$, there exists $s_m \in [0, 4)$ so that for $s \geq s_m$, $$c_m(s)
=  2m\int_{0}^{\infty} \left[\kappa_{mm}(r)-1\right]r^{2m-1-s}dr > 0 . $$
\end{prop}

We do not try to find a good value of $s_m$, which probably requires a computer assisted
calculation, but are content with a simple argument that nevertheless clarifies the structure
of $\kappa_{mm}(r)$.

\subsection{The higher codimension case}

The methods extend readily to systems of $k \leq m$ polynomials (or sections), whose zero set is a complex
submanifold of codimension $k$, so we also include this case.  If  $A\subset M$ is a complex subvariety of codimension $k\leq m-1$,  the Riesz functions have no point mass on the diagonal, and we define the continuous
Riesz energies as, \begin{equation}\label{def2}
\mathcal E_{s}(A)=\int\int_{A\times A} G_s(z,w)dV_k(z)dV_k(w)
 \end{equation}where  $dV_k=\frac{\omega^{m-k}}{(m-k)!}$ is the volume form on $A$.

The zero locus  $Z_{s_1^N,\cdots, s^N_k}$  of $k$  holomorphic random sections is an analytic subvariety of codimension $k$ where $k<m$.
We  define,
\begin{equation}\label{defsd2}
E_{\mu^N_h}\mathcal E_{s}=\int_{(H^0(M,L^N)^k}\mathcal E_{s}(Z_{s_1^N,\cdots, s^N_k})d\mu_h^N(s)=\int_{(\C^{d_N})^k}\mathcal E_{s}(Z_{s_1^N,\cdots, s^N_k})\frac{e^{-|c|^2}}{\pi^{d_N}}dc,
\end{equation}
where \begin{equation}\label{msthi}\mathcal E_{s}(Z_{s_1^N,\cdots, s^N_k})=\int_{Z_{s_1^N,\cdots, s^N_k}\times Z_{s_1^N,\cdots, s^N_k}}G_s(z,w)dV_k(z)dV_k(w)\end{equation}
We then have,

\begin{thm}  \label{con} With the same  assumptions as in Theorem \ref{dis},
\begin{itemize}
\item If $0< s<2(m-k)$, we have,
\begin{align*}E_{\mu_h^N}\mathcal E_{s}=&\left(\frac{\pi^{m-k}}{(m-k)!}\right)^2\left[a_1(h,s)N^{2k}+\sum_{j=2}^{p-1}w_j(k)a_j(h,s)N^{2k-j}\right]\\&+d_m(k,s)N^{2k-m+\frac{s}{2}}+O(N^{2k-m+\frac{s-1}{2}}(\log N)^{m-\frac{s}{2}})\end{align*} where $p =[m-\frac{s}{2}]+1$ and $d_m(k,s)$ is the integral defined in (\ref{eq:newejffffjjjq}).

 \item If $s=0$, we have,\begin{align*}E_{\mu_h^N}\mathcal E_{0}&=\left(\frac{\pi^{m-k}}{(m-k)!}\right)^2 \left[a_1(h,0)N^{2k}+w_2(k)a_2(h,0)N^{2k-2}+\cdots+w_m(k)a_m(h,0)N^{2k-m}\right]\\& -d_m(k)N^{2k-m}+e(k)N^{2k-m }\log \sqrt N
+O(N^{2k-m-\frac{1}{2}}(\log N)^{m+1})\end{align*} where  $d(k)$ and $e(k)$ are constants depending on $k$.
 \end{itemize}
where $a_j(h,s)$ are defined in Theorem \ref{dis} and $w_j(k)$ is the weight, in particular, $w_2(k)=\frac k m$.

\end{thm}
\begin{rem}Note that conditions of $s$ in the  two theorems are different. This  is  because the behavior
of the  pair correlation functions of zeros   near the origin are different when $k\neq m$ and $ k=m$ respectively.
\end{rem}

\subsection{Fekete points}

Although we do not have any results on it, we should point out that the separation and equidistribution of points
in dimensions $m > 1$ are often measured by the energies of Leja and Zaharjuta,
 \begin{equation} \label{DELTA} \log \Delta_N(z_1, \dots, z_{d_N}) = \log \det \left(B_N(z_i, z_j) \right) \end{equation}  where $B_N$ is the Bergman kernel for degree $N$
polynomials (the kernel of the orthogonal projection onto the degree $N$ polynomials).  Here, $d_N $ is
the dimension of the space of polynomials of degree $N$.  In dimension one,
this is a Vandermonde type determinant whose logarithm  is closely related to the Green's energy of the
configuration. But in higher dimensions, the Green's energy and  $ \log  \det \left(B_N(z_i, z_j) \right)$  are no longer
known to be closely related and probabaly are not.
Moreover, one is forced to choose different numbers of points for the two different energies:  for \eqref{DELTA}
one chooses $d_N$ points,
but in the Green's energy problem in the point case the  number of zeros of the system is $c_1(L^N)^m > d_N$.
We refer to \cite{BBWN} for recent result on Fekete points in higher dimensions.

\subsection{Acknowledgements}

Finally, we thank B. Shiffman for helpful comment and for the computer graphics.

\section{Background}
In this section, we recall some definitions in \kahler geometry \cite{GH, Ze} and some properties about
zeros of random holomorphic sections of positive holomorphic line bundles \cite{BSZ1, BSZ2, BSZ3}.
\subsection{\kahler geometry}\

Let $(M,\omega)$ be an $m$-dimensional polarized compact \kahler manifold. Then there exists a polarization of $(M,\omega)$, i.e., there exists a positive holomorphic line bundle $(L, h)$ with smooth
Hermitian metric $h$ such that the
curvature form
$$\label{curvature}\Theta_h=-\ddbar
\log\|e_L\|_h^2\;,$$
satisfies the following condition $$\omega=\frac{\sqrt{-1}}{2}\Theta_h$$   Here, $e_L$ is a local non-vanishing
holomorphic section of $L$ over an open set $U\subset M$, and
$\|e_L\|_h=h(e_L,e_L)^{1/2}$ is the $h$-norm of $e_L$ \cite{GH}.

We denote by $H^0(M, L^{N})$ the space of global holomorphic
sections of
$L^N=L\otimes\cdots\otimes L$.   The metric $h$ induces Hermitian
metrics
$h^N$ on $L^N$ given by $\|s^{\otimes N}\|_{h^N}=\|s\|_h^N$.  We give
$H^0(M,L^N)$ the Hermitian inner product \eqref{innera}.

We define the Bergman kernel as the orthogonal projection from the $L^2$ integral sections to the holomorphic sections:
\begin{equation}\Pi_{N}(z,w): L^{2}(M,L^{N})\rightarrow H^{0}(M,L^{N})\end{equation}
with respect to the inner product $\langle \cdot, \cdot \rangle_{h^N}$.

Denote $\{S_1^N,\cdots, S_{d_N}^N\}$ as the orthonormal basis of $H^0(M,L^N)$ with respect to the
 inner product $\langle \cdot, \cdot \rangle_{h^N}$ where $d_N=\dim H^{0}(M,L^{N})$. Then along the diagonal, we have the following complete asymptotics \cite{L, Ze},
\begin{equation}\label{seg} \Pi_{N}(z,z)=\sum_{j=1}^{d_N}\|S_j^N\|^2_{h^N}=N^m+a_1(z)N^{m-1}+a_2(z)N^{m-1}\cdots
\end{equation} where each coefficient
$a_j (z)$ is a polynomial of the curvature and its covariant derivatives.  In particular, $a_1(z)$ is the scalar curvature of $\omega$.

\begin{example}In the case of $(\CP^m, \omega_{FS})$ with the line bundle $(\mathcal O(N), h_{FS})$ which is the $N$th power of the hyperplane line bundle $\mathcal O(1)$, the Bergman kernel is a constant along the diagonal \cite{BSZ2},
\begin{equation} \label{cp}\Pi_N(z,z)=\frac{(N+m)!}{N!}
\end{equation}
\end{example}
\subsection{Zeros of Random sections}\label{zeross}\

For Gaussian random holomorphic sections $s^N_1(z),\cdots, s^N_k(z)\in ((H^0(M, L^N))^k,d\mu^N_h)$ with $k\leq m$, we
denote by \begin{equation}Z_{s^N_1,\cdots,s^N_k}:=\left\{z\in M: s^N_1(z)=\cdots =s^N_k(z)=0\right\}.\end{equation}  It is a subvariety of codimension $k$.  The current of integration over the zero locus is defined by
\begin{equation}\left(Z_{s_1^N,\cdots,s_k^N},\phi\right)=\int_{Z_{s^N_1,\cdots,s^N_k}}\phi(z),\,\,\, \phi\in \mathcal{D}^{m-k,m-k}(M),\end{equation}
which we also write as
\begin{equation}\left(Z_{s^N_1,\cdots,s^N_m},\phi\right)=\int_{M}Z_{s_1^N,\cdots,s_k^N}\wedge \phi(z),\,\,\, \phi\in \mathcal{D}^{m-k,m-k}(M)\end{equation}
For the special point case of $k=m$,
\begin{equation}\left(Z_{s^N_1,\cdots,s^N_m},\phi\right)=\sum_{s^N_1(z)=\cdots =s^N_m(z)=0}\phi(z),\,\,\, \phi\in C(M).\end{equation}

We now equip the vector space $H^0(M,L^N)$ with the complex
Gaussian probability measure \eqref{HGM}  induced by $h$.
This
Gaussian
is characterized by the property that the $2d_N$ real variables $\Re
c_j,
\Im c_j$ ($j=1,\dots,d_N$) are independent random variables with mean 0
and
variance $\frac{1}{2}$; i.e.,
\begin{equation}E c_j = 0,\quad E c_j c_k = 0,\quad  E c_j \bar c_k = \delta_{jk}\,.\end{equation}
where $E$ denotes the expectation

We denote by  $Z_{s^N}$ the current of integration over the  zeros locus of a single section.
By the  Poincar\'{e}-Lelong formula. Then we have \cite{SZ1},
\begin{equation}\label{pl}E(Z_{s^N})=\frac{\sqrt {-1}}{2\pi}\partial\bar\partial\log \Pi_N+\frac{N}{\pi}\omega\end{equation}
It follows by \eqref{seg} that
\begin{equation}E(\frac{1}{N}Z_{s^N})\rightarrow \frac{\omega}{\pi}\,\,\, \mbox{as}\,\, \,N\rightarrow \infty\end{equation}

Define \begin{equation}\label{denote}E^N_{km}:=E(Z_{s^N_1,\cdots,s^N_k}),\,\,\, k\leq m\end{equation} as the expectation of zeros locus of $k$ random sections. Then  we have \cite{SZ1},
\begin{equation}\label{to}E^N_{km}=\left(E(Z_{s^N})\right)^k=\left(\frac{i}{2\pi}\d\dbar \log\Pi_N(z,z)+\frac{N}{\pi}\omega\right)^k.\end{equation}
By \eqref{seg} we get,
 \begin{equation}\label{tro}E(Z_{s^N_1,\cdots,s^N_k})=\frac{N^k}{\pi^k}\omega^k+O(N^{m-1})\end{equation}
for $N$ large enough.
As a direct consequence, we know the expected number of discrete zeros is $r_N=(c_1(L^N))^m$ when $k=m$.

\begin{example}In the case $\mathcal O(N)\to \CP^m$ equipped with the Fubini-Study metric,
$H^0(\CP^m, \mathcal O(N))$
is identified with the space of polynomials of degree $N$. We equip $H^0(\CP^m, \mathcal O(N))$ with complex Gaussian ensemble induced by Fubini-Study metric. Since the Bergman kernel is constant (\ref{cp}),
by Poincar\'{e}-Lelong formula (\ref{pl}), the distribution of zeros of Gaussian random polynomials $p^N$ satisfies the exact formula,
\begin{equation}E(\frac{1}{N}Z_{p^N})= \frac{\omega_{FS}}{\pi}\end{equation}
as $N$ large enough. And the expected number of zeros is $N^m$.
\end{example}

\subsection{\label{PCC} The pair correlation current}\

The pair correlation current is defined by,
 \begin{equation} \label{KN} K_{km}^N(z,w):=E\left(Z_{s^N_1,\cdots,s^N_k}(z)\otimes Z_{s^N_1,\cdots,s^N_k}(w)\right)\in \mathcal D^{2k,2k}(M\times M)\end{equation}
in sense that for any test form $\phi_1(z)\otimes\phi_2(w)\in \mathcal D^{m-k,m-k}(M)\otimes \mathcal D^{m-k,m-k}(M) $, we have the pairing,
\begin{equation}\left(K_{km}^N(z,w),\phi_1(z)\otimes\phi_2(w)\right)=E\left[\left(Z_{s^N_1,\cdots,s^N_k},\phi_1\right)\left(Z_{s^N_1,\cdots,s^N_k},\phi_2\right)\right]\end{equation}
In the special case when $k=m$, the correlation measures take the form
\begin{equation}\label{oxf2f}K_{mm}^N(z,w)=[\triangle]\wedge \left( E^N_{mm}(z)\otimes 1 \right)+\kappa^N(z,w)\omega_z^m\otimes\omega_w^m\end{equation}
where $[\triangle]$ denotes the current of the integration along the diagonal and $\kappa^N \in C^\infty(M\times M)$ and
$E_{mm}^N$ is the expectation $E(Z_{s^N_1,\cdots,s^N_m})$.

The correlation currents of two points $K_{km}^N$ are smooth forms away from the diagonal in $M\times M$ and has no mass on the diagonal
for $k<m$; in the case of $k=m$, $K_{mm}^N$ contains a delta-function along the diagonal \cite{SZ1}.

If $1\leq k\leq m-1$, the correlation current for $K_{km}^N$ is an $L^1$ current on $M\times M$ given by
\begin{equation}K_{km}^N=\left(-\partial_1\bar\partial_1\partial_2\bar\partial_2 Q_N+EZ_{s^N}\otimes EZ_{s^N}\right)^k\end{equation}
where $Q_N$ is the pluri-bipotential defined in \cite{SZ1}.

 In the case of $k=m$, \begin{equation}\label{off} K_{mm}^N=\left(-\partial_1\bar\partial_1\partial_2\bar\partial_2 Q_N+EZ_{s^N}\otimes EZ_{s^N}\right)^m|_{M\times M \setminus \bigtriangleup}+diag_* \left(EZ_{s^N} \right)^m\end{equation} where $\triangle$ is the diagonal in $M\times M$ and $diag: M\to M\times M$ is the diagonal map $diag(z)=(z,z)$.

 Furthermore, the pair correlation currents has the off-diagonal asymptotics \cite{SZ1},
 \begin{equation}\label{of2f}K_{km}^N=\left(EZ_{s^N}\otimes EZ_{s^N}\right)^k+O(N^{-q}), \,\,\, \mbox{for}\,\, r_g(z,w) \geq b\sqrt{\frac{\log N}{ N}}\end{equation}
where $b>\sqrt{q+2k+3}\geq 3$.

\subsubsection{ Universal rescaling property}\label{rescale}\

 We now   re-scale the correlation functions around a point $z_0 \in M$  in
normal coordinates by a factor of $\sqrt N$. In \cite{BSZ2}, the authors prove that the pair correlation currents have the universal rescaling asymptotics,

 \begin{equation}\label{current}K_{km}^N(\frac{z}{\sqrt N},\frac{w}{\sqrt N})
 = K^\infty_{km}(z, w)\left(\frac{i}{2\pi}\d\dbar |z|^2\right)^k\wedge \left(\frac{i}{2\pi}\d\dbar |w|^2\right)^k+O(\frac{1}{\sqrt N})\end{equation}
 where the error term is  $\frac{1}{\sqrt N}$ times a smooth volume form. The universal rescaling property is independent of the manifold $M$ and the line bundle $L$;
$ K^\infty_{km}$ depends only on the dimension
$m$ of the manifold and the codimension $k$ of the zero set. Furthermore, $K^\infty_{km}(z,w)$ is a function depending on the distance $r_g(z,w)$ and thus we rewrite \cite{BSZ3}, \begin{equation}\label{kappa}K^\infty_{km}(z, w)=\kappa_{km}(r),\,\,\,  r = |z - w|, \;\;1\leq k\leq m\end{equation}

The universal scaling limit $\kappa_{km}(r)$ is explicitly computed in \cite{BSZ2, BSZ3}.
In the case when $k =m$, one has the rather complicated formula ($v=e^{-r^2}$)
\begin{equation} \begin{array}{r}\kappa_{mm}(r) = \frac{m (1-{v}^{m+1} ) (1-v )+r^2(2m+2) ({v}^{m+1}-v )+r^4\left[{v}^{m+1}+{v}^{m}+ ( \{m+1\}v+1 )(v^m-v)/(v-1)\right]}{m (1-v)^{m+2}}\,,\\[8pt]  \end{array}\label{pointpaircor}\end{equation}
For small values of $r$, we have
\begin{equation} \label{leading} \kappa_{mm}(r)= \frac{m+1}{4}r^{4-2m} + O(r^{8-2m})\,,\qquad\mbox{as }\ r\to 0\,.\end{equation}
As mentioned above, the $\kappa_{11}(r) \sim r^2$ as $r \to 0$, exhibiting repulsion,
while $\kappa_{22}(r) \to \frac{3}{4}$ when $m = 2$. For $m \geq 3$, $\kappa_{mm}(r)$ blows up as $r \to 0$.

In the higher codimension case $k\neq m$,  \begin{equation}\label{subvvv} \kappa_{km}(r)\sim r^{-2k}, \,\,\,\,\mbox{as}\,\,\,\, r\to 0
\end{equation}
In both cases,
\begin{equation} \label{leading2} \kappa_{km}(r)=1 + O(r^4e^{-r^2})\,,\qquad\mbox{as }\ r\to \infty\,.\end{equation}
for any  $1\leq k\leq m$. Thus, the correlations are very short range on the $\frac{1}{\sqrt{N}}$
length scale. The constant $1$ would be the
correlation function for the Poisson spatial process (i.e. independent points chosen at random
from the normalized volume form).
Below are computer graphics of $\kappa_{mm}(r)$ in the cases $m =2,  3, 4$.
\begin{center} \label{m=2} 
\includegraphics[scale=0.5]{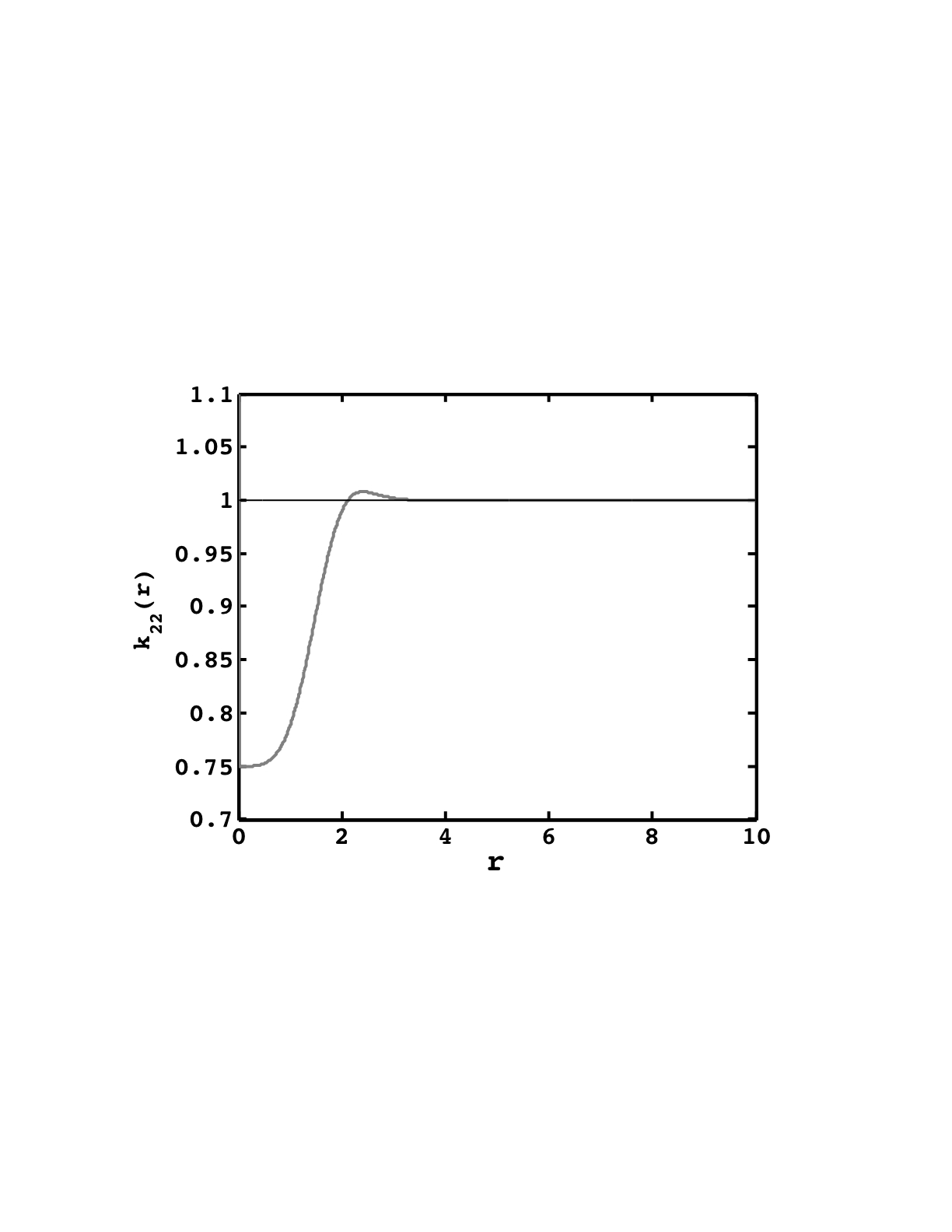} 
\text{m=3,4}\includegraphics[scale=0.5]{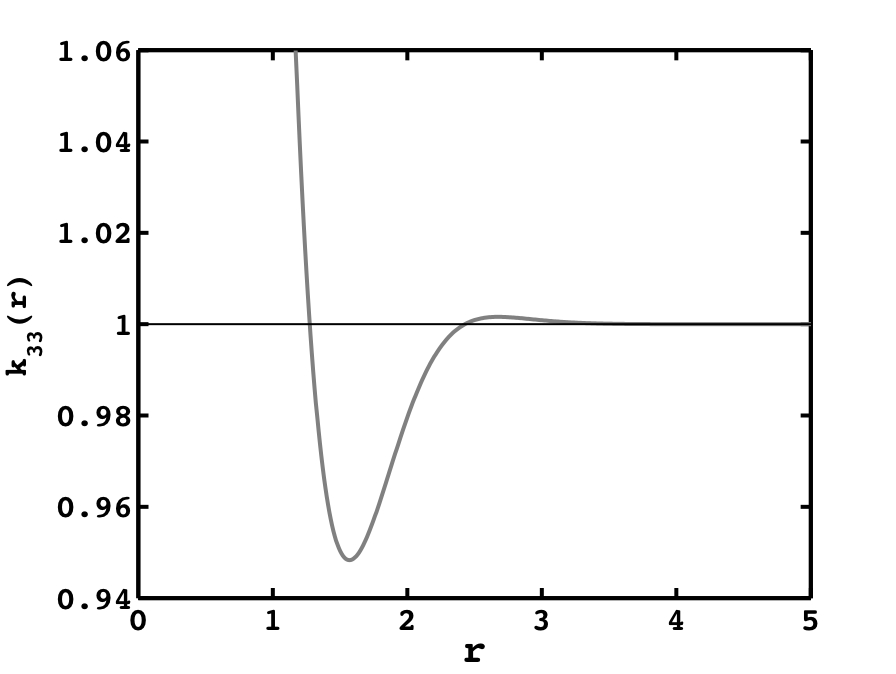}\includegraphics[scale=0.5]{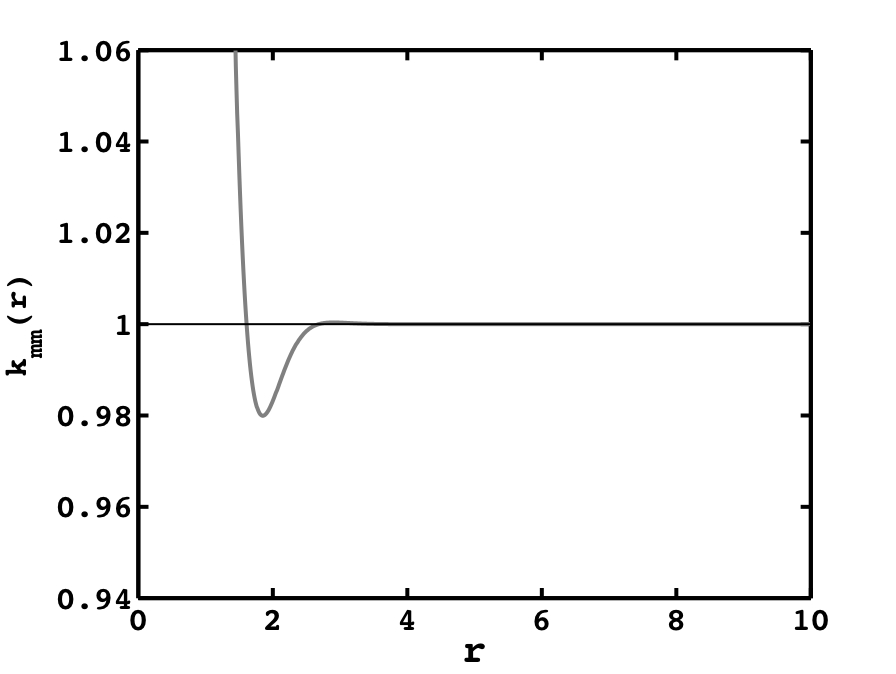}
 \end{center}

From the computer graphic. one sees that $\kappa_{33}(r) - 1$ and $\kappa_{44} - 1$ each
two zeros, both less than $4$, and each is positive elsewhere. In Proposition \ref{POS}  we show that $k_{mm}(r) - 1 > 0$
when $m \geq 3$ 
for $r^2 \geq 2m + 3$, but almost certainly it becomes positive at much smaller values of $r$.

\subsection{\label{CONFIG} Configuration spaces, empirical measures and random points}

Generic configurations of zeros $\zeta = \{\zeta^1, \dots, \zeta^{r_N}\}$ in the discrete (point) case are points in the punctured product
\begin{equation}\label{MrN}M_{r_N}=\{(\zeta^1,\dots,\zeta^{r_N})\in
\underbrace{M\times\cdots\times M}_{r_N}: \zeta^p\ne \zeta^q \ \ {\rm for} \
p\ne q\}/S_{r_N}, \end{equation}
where $r_N = (c_1(L^N))^m$ is the number of zeros of the system and $S_{r_N}$ is the symmetric
group on $r_N$ letters.  As in the case $m = 1$,
one has a map
$$\mu_{\zeta}:  M_{r_N} \to {\mathcal P}_M, \;\;\;\; \mu_{\zeta} = \frac{1}{r_N} \sum \delta_{\zeta_j}$$
from configurations to probability measures on $M$. We also have the  map \eqref{ZEROSMAP}
$$(s^N_1,\cdots,s^N_m) \in (H^0(M, L^N) )^m \to Z_{{s^N_1,\cdots,s^N_m}} \in M_{r_N} $$
from polynomial systems to their zeros, and the composition gives the empirical measure of zeros map
$$ (s^N_1,\cdots,s^N_m) \to \mu_{\zeta}. $$
The Gaussian product probability measure on $ (H^0(M, L^N) )^m$ pushes forward to a probability
measure on $M_{r_N}$ and the two-point correlation function \eqref{KN} is the expected value of
$\mu_{\zeta} \otimes \mu_{\zeta}$ with respect to this pushforward.  Note that the image in
$M_{r_N}$ of the zeros map \eqref{ZEROSMAP} has positive codimension growing linearly in $N$,
and therefore the induced measure on $M_{r_N}$ is very singular.

 By comparison, just choosing $r_N$ random points on $M$ from the normalized volume form \eqref{TILDEV}
corresponds to product measures on $M_{r_N}$, known as the binomial point process.
If we let $N \to \infty$ and rescale as above, we get the Poisson point process on $\C^m$ with intensity
given by the Euclidean measure. Its correlation function is then the constant function $1$. Hence
$\kappa_{mm} - 1$ repressents the deviation of the zero point correlations from the Poisson case.

The Riesz energies of configurations \eqref{def}  are functions on $M_{r_N}\times M_{r_N}$ which blow up on the `large diagonals'
where some pair $\zeta^i = \zeta^j$ ($i \not= j$).  For polynomial systems, this large diagonal
defines the discriminant variety $\dcal_{m, N}$ of full polynomial systems of degree $N$ in $m$
variables which have a double root.  Thus $\ecal_s$  defines  a simple  notion  of  `inverse
distance to $\dcal_{m, N}$ '.  The large expected value for $\ecal_s$ in dimensions $m \geq 3$ thus show
that configurations of zeros are often close to $\dcal_{m, N}$ by comparison with random points.

\section{Proof of Theorem \ref{dis}}
Now we compute the Riesz energy of the discrete set $Z_{s_1^N,\cdots,s_m^N}$ which is the zero set of
$m$ random sections. In the proof, $c$ is a constant, but may be different from line to line.  First we need the following Lemma,
\begin{lem}\label{k1}We have complete asymptotics,
\begin{equation}\label{ddd}\begin{aligned}&\int_{M\times M}G_s(z,w)E_{mm}^N(z)\wedge E_{mm}^N(w)\\&=a_1(h,s)N^{2m}+a_2(h,s)N^{2m-2}+a_3(h,s)N^{2m-3}+\cdots\end{aligned}\end{equation}
where each coefficient $a_j(h, s)$ is a constant depending on $s$ and the curvature of $h$, in particular, the first two terms are given in (\ref{coe}).
\end{lem}
\begin{proof}
From (\ref{to}), $E^N_{mm}(z)$ which is the expectation of discrete zeros of $m$ random sections is given by the explicit formula,
\begin{equation}
E^N_{mm}(z)=\left(\frac{i}{2\pi}\partial\bar\partial\log \Pi_N(z,z)+\frac
N\pi\omega(z)\right)^m,
\end{equation}
First by \eqref{seg},  we have,  \begin{equation}
\log\Pi_N(z,z)=m\log
N+a_1(z)N^{-1}+\left[a_2(z)-\frac{1}{2}a_1^2(z)\right]N^{-2}+\cdots
\end{equation} where $a_1=S(z)$ is the scalar curvature of $\omega$ and each $a_j(z)$ is a polynomial of curvature and its covariant derivative.  Thus we get the following complete asymptotics of $E^N_{mm}(z)$, \begin{equation}\label{K1}\begin{aligned}
E^N_{mm}(z)= &\left(  \frac\omega \pi N+ \frac{i\partial\bar\partial
S(z)}{\pi}N^{-1}+\cdots\right)^m\\ =&N^m\left(\frac{\omega}{\pi}\right)^m+mN^{m-2}\left(\frac{\omega}{\pi}\right)^{m-1}\wedge \frac{ i\partial\bar\partial S(z)}{\pi}+\cdots\end{aligned}
\end{equation} which implies that the integration (\ref{ddd}) admits
complete asymptotics. Each  $a_j(h,s)$ is well defined since $G_s \in L^1(M \times M)$ for $ 0\leq s<2m$.

We can in fact compute each coefficient from the asymptotic expansion of $E_{mm}^N$. From the complete expansion of (\ref{K1}), and  the identity  $$m(i\partial\bar\partial S)\wedge \omega^{m-1}=(\Delta_\omega S)\omega^m$$
we obtain \eqref{coe}
\end{proof}
Now we are ready to prove Theorem \ref{dis}.

\subsection{The case of $0<s<2m$}\

We first prove the case $0<s<2m$, the logarithm case when $s=0$ can be derived by modifications. By the definition of the discrete random Riesz energy, we have
\[ \begin{aligned}E_{\mu_h^N}\mathcal E_{s}&
 =E_{\mu_h^N}\left(G_s(z,w),Z_{s_1^N,\cdots,s_m^N}(z) \otimes Z_{s_1^N,\cdots,s_m^N}(w)- Z_{\Delta}(z)\right)\\
&=\left(G_s(z,w),K^N_{mm}(z,w)-[\Delta]\wedge
\left(E^N_{mm}(z)\otimes 1\right)\right)\\
& = \int_{M \times M} G_s(z,w) \left[K^N_{mm}(z,w)-[\Delta]\wedge
\left(E^N_{mm}(z)\otimes 1\right)\right]\\
&=\int_M\int_{r_g(z,w) \leq \frac{b \sqrt{\log N}}{\sqrt{N}}}
G_s(z,w)\left[K^N_{mm}(z,w)-[\Delta]\wedge \left(E^N_{mm}(z)\otimes 1\right)\right]\\
&+\int_M\int_{r_g(z,w) \geq \frac{b \sqrt{\log N}}{\sqrt{N}}}
G_s(z,w)K^N_{mm}(z,w)\\
\end{aligned}\]
By Lemma \ref{k1}, we rewrite $E_{\mu_h^N}\mathcal E_{s}$ as:

\begin{align*}
E_{\mu_h^N}\mathcal E_{s}&
 =\int_M\int_{r_g(z,w) \leq \frac{b \sqrt{\log N}}{\sqrt{N}}}
G_s(z,w)\left[K^N_{mm}(z,w)-[\Delta]\wedge \left(E^N_{mm}(z)\otimes 1\right)\right]\\
&+\int_M\int_{r_g(z,w) \geq \frac{b \sqrt{\log N}}{\sqrt{N}}}
G_s(z,w)K^N_{mm}(z,w)-\int_{M\times M}G_s(z,w)E_{mm}^N(z)\wedge E_{mm}^N(w)\\
&+a_1(h,s)N^{2m}+\sum_{j=2}^{p-1}a_j(h,s)N^{2m-j}+O(N^{2m-p})\\
&=\int_M\int_{r_g(z,w) \leq \frac{b\sqrt{\log N}}{\sqrt{N}}}
G_s(z,w)\left[K^N_{mm}(z,w)-[\Delta]\wedge
\left(E^N_{mm}(z)\otimes1\right)-E_{mm}^N(z)\wedge E_{mm}^N(w)\right]\\
&+\int_M\int_{r_g(z,w) \geq \frac{b \sqrt{\log N}}{\sqrt{N}}}
G_s(z,w)\left[K^N_{mm}(z,w)-E^N_{mm}(z)\wedge E^N_{mm}(w)\right]\\
&+a_1(h,s)N^{2m}+\sum_{j=2}^{p-1}a_j(h,s)N^{2m-j}+O(N^{2m-p})\\
&:=I+II+III
\end{align*}
where $p$ is a positive integer which will be chosen later on.

The following Lemma asserts that $II$ could be ignored.
\begin{lem}\label{toe} $$II=O\left(N^{-q+\frac{s}{2}}\right)$$ where $q$ is a positive integer which can be chosen large enough.
\end{lem}
\begin{proof}
If $r_g(z,w) \geq \frac{b \sqrt{\log N}}{\sqrt{N}}$, by the definition of Riesz functions, we have, $$|G_s(z,w)|\leq c\left(\frac{N}{\log N}\right)^{\frac{s}{2}}$$ as $N$ large enough. If we choose $k=m$ in the estimate (\ref{of2f}),
$$K_{mm}^N(z,w)-E^N_{mm}(z)\wedge E^N_{mm}(w)=O(N^{-q}) \,\,\,\mbox{on} \,\,\,r_g(z,w) \geq \frac{b \sqrt{\log N}}{\sqrt{N}} $$
where $b>\sqrt{q+2k+3}$.
If we combine these two estimate, the term $II$ reads,
$$\int_M\int_{r_g(z,w) \geq \frac{b \sqrt{\log N}}{\sqrt{N}}}
G_s(z,w)\left[K^N_{mm}(z,w)-E^N_{mm}(z)\wedge E^N_{mm}(w)\right]=O(N^{-q+\frac{s}{2}})$$
which completes the proof.
\end{proof}

Now we turn to estimate $I$ to finish the proof of Theorem \ref{dis}.
\begin{lem}\label{relmmma2}
$$I=c_m(s)N^{m+\frac{s}{2}}+O\left(N^{m+\frac{s-1}{2}}(\log N)^{m-\frac{s}{2}}\right) \,\,\, \text{$0<s<2m, s<4$,}$$
where $c(s)$ is a constant given explicitly in the proof.
\end{lem}

\begin{proof}
We change variable
$$w=z+\frac u{\sqrt N}$$
 Then by the rescaling property of (\ref{current}), we have,
\begin{align}\label{become H11}
K^N_{mm}(z,w)-[\Delta]\wedge (E^N_{mm}(z)\otimes1)=&K^N_{mm}(z,z+\frac
u{\sqrt N})-[\Delta]\wedge
(E^N_{mm}(z)\otimes1)\notag\\
=&N^m\kappa_{mm}(|u|)\left(\frac{\omega}{\pi}\right)^m\wedge(\frac
i{2\pi}\partial\bar\partial |u|^2)^m+O(N^{m-\frac{1}{2}})
\end{align}
where $\kappa_{mm}(|u|)$ is defined in (\ref{kappa}) with $k=m$.

By the estimate of (\ref{K1}),
\begin{equation}\label{kk}
E^N_{mm}(z)\wedge E^N_{mm}(w)=N^{2m}\left(\frac{\omega(z)}{\pi}\right)^m\wedge\left(\frac{\omega(w)}{\pi}\right)^m +O(N^{2m-2}).
\end{equation}
By choosing the normal coordinate, if $\varphi$ is the K\"ahler potential, then we have $$\varphi(z+\frac{u}{\sqrt
N})=\frac{|u|^2}{2N}+O(\frac {|u|^3}{N^{\frac 32}}).$$ Since
$u=\sqrt N(w-z)$, we have
\begin{align*}\partial_w\bar\partial_w\varphi(w)&=\partial_u\bar\partial_u\varphi(z+\frac u{\sqrt
N})=\partial_u\bar\partial_u\frac{|u|^2}{2N}+O(N^{-\frac{3}{2}})
\end{align*} Therefore,
\begin{equation}\label{time12}
\frac N\pi\omega(w)=\frac i{2\pi}\partial_u\bar\partial_u|u|^2+O(N^{-\frac{1}{2}})
\end{equation}

Thus we rewrite (\ref{kk}) near the diagonal when $|u|_g\leq b\sqrt {\log N}$ as,
\begin{equation}{\label{K-K}}
E^N_{mm}(z)\wedge E^N_{mm}(z+\frac u{\sqrt N})=N ^m \left(\frac{\omega}{\pi}\right)^m\wedge
\left(\frac i {2\pi}\partial\bar\partial |u|^2 \right)^m+ O(N^{m-\frac{1}{2}})
\end{equation}
Recall the definition of $I$,
$$I=\int_M\int_{r_g(z,w) \leq \frac{b\sqrt{\log N}}{\sqrt{N}}}
G_s(z,w)\left[K^N_{mm}(z,w)-[\Delta]\wedge
\left(E^N_{mm}(z)\otimes1\right)-E_{mm}^N(z)\wedge E_{mm}^N(w)\right]$$ Now we apply identities (\ref{become H11}), (\ref{K-K}) to  rewrite,
\begin{equation} \label{iesti}\begin{aligned}
I=N^m\int_M\int_{0<|u|_g \leq b \sqrt{\log N}}
&G_s(z,z+\frac u{\sqrt N})\left[\kappa_{mm}(|u|)-1+ O(N^{-\frac{1}{2}} )\right]\\
&\cdot\left(\frac{\omega}{\pi}\right)^m\wedge \left(\frac{i}{2\pi}\partial\bar\partial |u|^2\right)^m \end{aligned}
\end{equation} The error term $O(N^{-\frac{1}{2}})$ comes from estimate (\ref{become H11}) and (\ref{K-K}).

Thus we further rewrite (\ref{iesti}) as,
\begin{equation}\label{dk0}\begin{aligned}I=N^m\int_M\int_{|u|_g \leq b \sqrt{\log N}}
&G_s(z,z+\frac u{\sqrt N})\left[\kappa_{mm}(|u|)-1\right]\\
&\cdot\left(\frac{\omega}{\pi}\right)^m\wedge \left(\frac{i}{2\pi}\partial\bar\partial |u|^2\right)^m + R_s
\end{aligned}\end{equation} Denote $I=\tilde I+R_s$.
The estimate of the error term $R_s$ is as follows,
\begin{align*}
R_s=O(N^{m-\frac{1}{2}}) \int_M\int_{|u|_g \leq b \sqrt{\log N}}
G_s(z,z+\frac u{\sqrt N})\left(\frac{\omega}{\pi}\right)^m\wedge \left(\frac{i}{2\pi}\partial\bar\partial |u|^2\right)^m
\end{align*}If we choose the polar coordinate, we have
\begin{align*}&\left|\int_M\int_{0\leq|u|_g \leq b \sqrt{\log N}}
G_s(z,z+\frac u{\sqrt N})\left(\frac{\omega}{\pi}\right)^m\wedge \left(\frac{i}{2\pi}\partial\bar\partial |u|^2\right)^m\right|\\
&\leq c\int_{0\leq |u|_g \leq b \sqrt{\log N}}\frac{1}{|\frac{u}{\sqrt N}|_g^s}|u|^{2m-1}d |u|\leq c N^{\frac{s}{2}}(\log N)^{m-\frac{s}{2}}\end{align*}
In the last step, we apply the estimate of the geodesic distance function $r_g(z,w)$ near the origin,  \begin{equation}\label{sidnc}r_g(z,z+\frac u{\sqrt N})=\frac{|u|}{\sqrt N}+O(N^{-\frac{3}{2}})\end{equation}
in the normal coordinate.
 Thus, \begin{equation}R_s=O\left(N^{m+\frac{s-1}{2}}(\log N)^{m-\frac{s}{2}}\right)\,\,\, \text{if $0<s<2m$,}
\end{equation}

The estimate of the integration $\tilde I$ is as follows, first we rewrite,
\begin{align*}\tilde I=& N^{m+\frac{s}{2}}\int_M\int_{0\leq|u| \leq \infty} \frac{1}{|u|^s}\left[\kappa_{mm}(|u|)-1\right]\left(\frac{\om}{\pi}\right)^m\wedge \left(\frac{i}{2\pi}\partial\bar\partial |u|^2\right)^m\\
&+\mbox{lower order terms}\\
:=&A_s+B_s
\end{align*} by applying the expansion of the geodesic distance function (\ref{sidnc}).
 We use  polar coordinates to compute the leading term,
\begin{equation}\label{eq:newejjjjq}
	\begin{aligned}A_s=& N^{m+\frac{s}{2} }\int_M\int_{0\leq|u| \leq \infty} \frac{1}{|u|^s}\left[\kappa_{mm}(|u|)-1\right]\left(\frac{\om}{\pi}\right)^m\wedge \left(\frac{i}{2\pi}\partial\bar\partial |u|^2\right)^m\\
=& N^{m+\frac{s}{2} }\frac{V_{2m-1}}{\pi^m}\int_{0\leq|u| \leq \infty}\frac{1}{|u|^s}\left[\kappa_{mm}(|u|)-1\right]|u|^{2m-1}d|u|\\
=&\left(2m\int_{0\leq|u| \leq \infty}\left[\kappa_{mm}(r)-1\right]r^{2m-1-s}dr\right)N^{m+\frac{s}{2} }\\
:=&c_m(s)  N^{m+\frac{s}{2} }
\end{aligned}
\end{equation} where $V_{2m-1}$ is the volume of $(2m-1)$-dimensional unit sphere.

The integrand in $c(s)$ is always integrable on the tail $a<|u|<\infty$ as $a$ large enough, this is easy to see because of
asymptotic property $\kappa_{mm}(r)=1+O(r^4e^{-r^2})$
as $r\to \infty$. But we need some conditions to ensure $c(s)$ to be well defined near the origin.
Since we know $\kappa_{mm}(r)\to \frac {m+1} 4r^{4-2m}$ as $r\to 0$, it's easy to check that $s<4$ is a necessary condition.

 Moreover, we have the easy estimate,
$$B_s=O(N^{m+\frac {s} 2-1})$$
Thus the growth of the error term $R_s$ will dominate the growth of $B_s$, and $$I=A_s+R_s=c_m(s)  N^{m+\frac{s}{2} }+O\left(N^{m+\frac{s-1}{2}}(\log N)^{m-\frac{s}{2}}\right)$$
which completes the proof.\end{proof}

We now finish the proof of the main result Theorem \ref{dis} for $0<s<2m$.
\begin{proof} For the case $0<s<2m$ and $s<4$,
recall the decomposition of $E_{\mu_{h}^N}\mathcal E_{s}$,
$$E_{\mu_{h}^N}\mathcal E_{s}=a_1(h,s)N^{2m}+\sum_{j=2}^{p-1}a_j(h,s)N^{2m-j}+O(N^{2m-p})+I+II$$
Now we apply the estimates in Lemma \ref{toe} and Lemma \ref{relmmma2}, we have,
\begin{align*}E_{\mu_{h}^N}\mathcal E_{s}
=&a_1(h,s)N^{2m}+\sum_{j=2}^{p-1}a_j(h,s)N^{2m-j}+O(N^{2m-p})\\
&+c_m(s)N^{m+\frac{s}{2}}+O\left(N^{m+\frac{s-1}{2}}(\log N)^{m-\frac{s}{2}}\right)+ O(N^{-q+\frac{s}{2}})
  \end{align*} The smallest integer $p$ such that $2m-(p-1)>m+\frac{s}{2}$ will ensure the above expansion is well defined, i.e, $p=[m-\frac{s}{2}]+1$. The last term is negligible for  $q$ large enough. Thus we rewrite,
  \begin{align*}E_{\mu_{h}^N}\mathcal E_{s}
=&a_1(h,s)N^{2m}+\sum_{j=2}^{p-1}a_j(h,s)N^{2m-j}\\
&+c_m(s)N^{m+\frac{s}{2}}+O\left(N^{m+\frac{s-1}{2}}(\log N)^{m-\frac{s}{2}}\right)
  \end{align*}
  which completes the proof.  \end{proof}
\subsection{The case of $s=0$}\

Now we sketch the proof for the case of $s=0$ in Theorem \ref{dis}.

First, as before, we rewrite,  \begin{align*}
E_{\mu_h^N}\mathcal E_{0}
&=-\int_M\int_{r_g(z,w) \leq \frac{b\sqrt{\log N}}{\sqrt{N}}}
\log r_g(z,w) \left[K^N_{mm}(z,w)-[\Delta]\wedge
\left(E^N_{mm}(z)\otimes1\right)-E_{mm}^N(z)\wedge E_{mm}^N(w)\right]\\
&-\int_M\int_{r_g(z,w) \geq \frac{b \sqrt{\log N}}{\sqrt{N}}}
\log r_g(z,w) \left[K^N_{mm}(z,w)-E^N_{mm}(z)\wedge E^N_{mm}(w)\right]\\
&+a_1(h,0)N^{2m}+\sum_{j=2}^{p-1}a_j(h,0)N^{2m-j}+O(N^{2m-p})\\
&:=I+II+III
\end{align*}
Follow the steps in Lemma \ref{toe}, $$II=O(N^{-q}\log N) \,\,\,\text{for $q$ large enough}$$
since $|G_0|$ is bounded by $\log N$ if $r_g(z,w) \geq \frac{b \sqrt{\log N}}{\sqrt{N}}$. Thus the term $II$ can be
ignored in the final expansion.

 The statement in Theorem \ref{dis} about logarithm Riesz energies follows directly by the following lemma.
 \begin{lem}\label{lemma0} \begin{equation}I=-N^{m }\log \sqrt N-c_mN^m
+O(N^{m-\frac{1}{2}}(\log N)^{m+1})\end{equation}
where $c$ is a constant given by (\ref{eq:dddrreneweq}).
 \end{lem}
 \begin{proof}Follow the proof of (\ref{dk0}) in Lemma \ref{relmmma2}, we rewrite $I=\tilde I+R_0$ as follows, \begin{equation}\begin{aligned}I=N^m\int_M\int_{|u|_g \leq b \sqrt{\log N}}
&G_0(z,z+\frac u{\sqrt N})\left[\kappa_{mm}(|u|)-1\right]\\
&\cdot\left(\frac{\omega}{\pi}\right)^m\wedge \left(\frac{i}{2\pi}\partial\bar\partial |u|^2\right)^m + R_0
\end{aligned}\end{equation}
where \begin{align*}
R_0=O(N^{m-\frac{1}{2}}) \int_M\int_{|u|_g \leq b \sqrt{\log N}}
G_0(z,z+\frac u{\sqrt N})\left(\frac{\omega}{\pi}\right)^m\wedge \left(\frac{i}{2\pi}\partial\bar\partial |u|^2\right)^m
\end{align*}
Using  the geodesic distance estimate (\ref{sidnc}), we have,
 $$\left|\int_{0\leq |u|_g \leq b \sqrt{\log N}}\log \left({\frac{|u|_g}{\sqrt N}}\right)|u|^{2m-1}d |u|\right|\leq c(\log N)^{m+1}$$
Thus,
$$R_0=O\left(N^{m-\frac{1}{2}}(\log N)^{m+1}\right)\,\,\, \text{if $s=0 $.}$$
We further rewrite $\tilde I=A_0+B_0$ by applying the  expansion of the geodesic distance function (\ref{sidnc}) again,
\begin{equation}\label{eq:nedddweq}
	\begin{aligned}A_0=&- N^{m }\int_M\int_{0\leq|u| \leq \infty} \log\left(\frac{|u|}{\sqrt N}\right)\left[\kappa_{mm}(|u|)-1\right]\left(\frac{\om}{\pi}\right)^m\wedge \left(\frac{i}{2\pi}\partial\bar\partial |u|^2\right)^m\\
=&N^{m }\log \sqrt N\int_{0\leq|u| \leq \infty}\left[\kappa_{mm}(|u|) -1\right]\left(\frac{i}{2\pi}\partial\bar\partial |u|^2\right)^m\\
&-\left(2m\int_{0\leq|u| \leq \infty}\log r\left[\kappa_{mm}(r)-1\right]r^{2m-1}dr\right)N^{m }
\end{aligned}
\end{equation}
Now claim:
\begin{equation}\label{s0}\int_{0\leq|u| \leq \infty}\left[\kappa_{mm}(|u|)-1\right]\left(\frac{i}{2\pi}\partial\bar\partial |u|^2\right)^m=-1\end{equation}
We calculate the integration directly by the rescaling property: 
\begin{equation}\label{eqau}\begin{aligned}&\int_M \left(\frac{\omega}{\pi}\right)^m\int_{0\leq|u| \leq \infty}\left[\kappa_{mm}(|u|)-1\right]\left(\frac{i}{2\pi}\partial\bar\partial |u|^2\right)^m\\
=&\lim_{N\to\infty}\frac 1 {N^m}\int_M\int_{r_{g}(z,w)\leq \frac{b\sqrt{\log N}}{\sqrt N}}\left\{K^N_{mm}(z,w)-[\Delta]\wedge
E^N_{mm}(w)
-E_{mm}^N(z)\wedge E_{mm}^N(w)\right\}\\
=&\lim_{N\to\infty}\frac 1 {N^m}\left(\int_M\int_M\cdots -\int_M\int_{r_g(z,w)\geq \frac{b\sqrt{\log N}}{\sqrt N}}\cdots\right)\\
\end{aligned}\end{equation}
The claim follows by the following two calculations: By the definition of correlation currents, the first term reads, 
$$\begin{aligned}&\lim_{N\to\infty}\frac 1 {N^m}\int_M\int_{M}\left\{K^N_{mm}(z,w)-[\Delta]\wedge
E^N_{mm}(w)
-E_{mm}^N(z)\wedge E_{mm}^N(w)\right\}\\=&\lim_{N\to\infty}\frac 1 {N^m}(N^{2m}-N^m-N^{2m})=-1
\end{aligned}$$
The second term reads, 
$$\begin{aligned}&\lim_{N\to\infty}\frac 1 {N^m}\int_M\int_{r_{g}(z,w)\geq \frac{b\sqrt{\log N}}{\sqrt N}}\left(K^N_{mm}(z,w)
-E_{mm}^N(z)\wedge E_{mm}^N(w)\right)\\&=\lim_{N\to\infty}O(N^{-q-m})=0
\end{aligned}$$ where we applied estimate (\ref{of2f}) for $q$ large enough.  
Thus \begin{equation}
	\begin{aligned}A_0=-N^{m }\log\sqrt N
-c_mN^{m }
\end{aligned}
\end{equation} where $c_m$ is given by \eqref{eq:dddrreneweq}.
The integral above is also well defined by the asymptotic properties of $\kappa_{mm}(r)$.
And it's also easy to check, $$B_0=O(N^{m-1})$$
    The growth of $B_0$ is less than the
growth of $R_0$, thus the error terms $B_0$ can be absorbed in the error term $R_0$. Thus in the case of $s=0$,
$$I=A_0+R_0=-N^{m }\log \sqrt N-c_mN^m
+O(N^{m-\frac{1}{2}}(\log N)^{m+1})$$
which completes the proof.
\end{proof}

\section{Proof of Proposition \ref{POS}}

The proof is  almost trivial but we include it to clarify the structure of 
 $\kappa_{mm}(r) - 1$ \eqref{pointpaircor}. The main point is that $c_m(s)$ is a meromorphic function
with  poles at $s = 4, 2m$. When $m \geq 2$, the only pole for $s \in [0, 4]$ is at $s = 4$. When $m \geq 3$
the residue is positive, while for $m = 2$ it is negative. When $m = 1$ it also  has a pole at $s = 2 \in [0, 4]$ 
with a negative residue. The Proposition just uses these facts to draw conclusions on positivity near enough
to the poles.

The results are illustrated in the following computer graphics in the cases $m = 2, 3, 4$. The graphs
for $m \geq 4$ are rather similar to that for $m = 3, 4$.  The integral in
\eqref{POSITIVE}  or more precisely $\int_0^{\infty} (\kappa_{mm}-1)r^{2m-1-s}dr$ is evaluated numerically over the interval $[0, 10000]$ so that positivity of the numerical
integration implies positivity of the true integral. As the  figures show, $c_m(s) < 0$ for small $s$, in
which case the random zeros have smaller Riesz $s$-energy on average than general random point. When $m = 2$
it appears always to be negative, so that random zeros have less energy than random points on average. When $m \geq 3$,
$c_m(s) > 0$ for $s $ somewhat larger than $s = 1$ and the rise is very steep as  $s \to 4$.
 As $m$ increases, the clumping of zeros increases and hence the Riesz energies increase.

\begin{center} \text{m=2}
\includegraphics[scale=0.4]{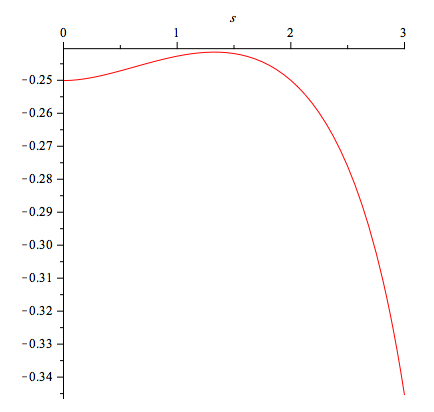}
\includegraphics[scale=0.4]{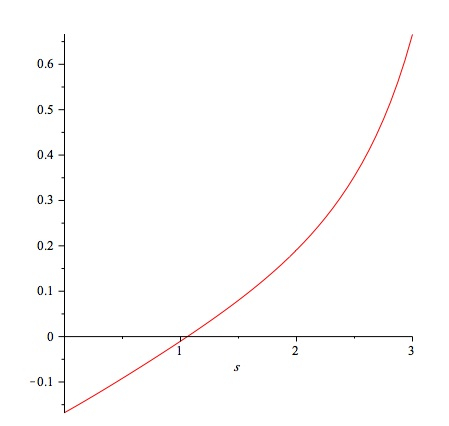} 
\includegraphics[scale=0.4]{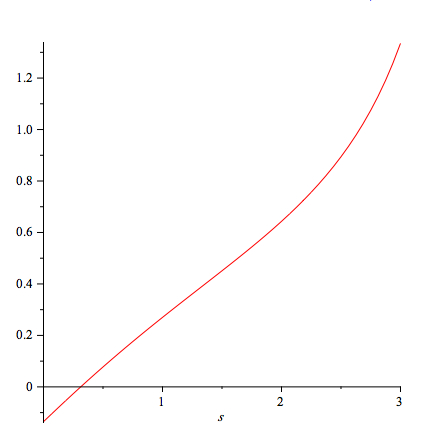} 
\text{m=3,4}
 \end{center}

Here are the details:
We observe that the numerator of $\kappa_{mm}$  consists of three terms
\begin{equation}
\left\{ \begin{array}{l}  \kappa_{mm}^I(r) = m (1-{v}^{m+1} ) (1-v ), \\ \\
\kappa_{mm}^{II}(r) = r^2(2m+2) ({v}^{m+1}-v ), \\ \\
\kappa_{mm}^{III}(r) = r^4\left[{v}^{m+1}+{v}^{m}+ ( \{m+1\}v+1 )(v^m-v)/(v-1)\right].   \end{array}\label{pointpaircora} \right. \end{equation}
Since $v < 1$ only the second is negative.  Furthermore the second and third
terms are divisible by $v = e^{-r^2}$ and therefore the first term dominates for large $r$.
We claim that there exists $r_0(m) $ so that  $\kappa_{mm}(r) - 1$  is positive for  $r \geq r_0(m)$.
We do not attempt to find the best $r_0(m)$ but rather look for a simple one.

 We first
drop  the terms $v^m$ or higher powers from the numerator , simplifying
the numerator to $$N_0(m, x) : = m(1 - e^{-x}) -   (2m + 2)  x e^{-x}  + x^2[(m + 1) e^{-x} + 1] \frac{e^{-x}}
{1 - e^{-x}}$$   with $x  = r^2$.
We would like to find $x$ so that $N_0(m, x) > m (1 - e^{-x})^{m + 2}$, i.e. so that
 $$  x^2[\frac{(m + 1) }{m}e^{-x} + \frac{1}{m}] \frac{e^{-x}}
{1 - e^{-x}}   - \frac{ (2m + 2)}{m}   x e^{-x}  - e^{-x}  > (1 - e^{-x})^{m + 2} - 1. $$
In fact, the left side is easily seen to be positive if $x \geq (2m + 3)$.
Returning to our original numerator, we want to show that $N_0(m, x)
> m $ if $x \geq (2m + 3) $. Given the above observations, it suffices to show that the missing
terms $- m v^{m + 1} (1 - v) + r^2(2m + 2) v^{m + 1}   > 0$ and this also holds when $x \geq 3$. Consequently,
$\kappa_{mm}(r)  - 1> 0$ if $x \geq (2m +3) . $ Thus, we can take $r_0(m) = \sqrt{(2m + 3)}.$
Hence, to prove that $c_m(s) > 0$ if suffices to show that
\begin{equation} \label{POSITIVE} \int_0^{\sqrt{2m + 3}}  \left[\kappa_{mm}(r)-1\right]r^{2m-1-s}dr > 0 \end{equation}
for $s \geq s_m$.
Indeed, the contribution for small $r$ has the form,
\begin{equation} \label{small} \int_{0}^{\epsilon} \left[\kappa_{mm}(r)-1\right]r^{2m-1-s}dr =
\frac{m + 1}{4} \frac{\epsilon^{4 - s} }{4 - s} - \frac{\epsilon^{2m - s}}{2m - s} + O_m (\epsilon^{8 - s}), \end{equation}
and only the first term has a pole as $s \to 4^-$ when $m \geq 3$. When $m = 2$ the integral of $-1$ 
also blows up at $s = 4$ and the coefficient of $\frac{1}{4 -s}$ is $-1/4$, so the integral is negative
for $s $ near $4$; when $m = 1$ the integral of $-1$  term tends to $-\infty$
at $s = 2$.  Hence in proving positivity we assume $m \geq 3$.   On the other hand, for $m \geq 3$
the integral of the left side of \eqref{POSITIVE} over the interval $[\epsilon, \sqrt{2m + 3}]$ is continuous
in $s \in [0, 4]$ and of course  has a uniform upper bound. Hence  \eqref{POSITIVE} holds for $s < 4$ sufficiently
close to $4$.

\section{Proof of Theorem \ref{con}}
 The proof of Theorem \ref{con} is similar to Theorem \ref{dis}. For simplicity, we only prove the case when $s\neq 0$, the result for logarithm energies can be derived by few modifications. We now sketch the proof as follows,
\begin{proof}
Suppose $Z_{s_1^N,\cdots, s^N_k}$ is the zero locus of $k< m$ holomorphic random sections, i.e., analytic subvariety of positive dimension $m-k$,
by the definition of the continuous random Riesz energy (\ref{defsd2}) and (\ref{msthi}), we have,
\[ \begin{aligned}E_{\mu_{h}^N}\mathcal E^N_{s} =&
E \left(G_s(z,w)dV_{k}(z)dV_{k}(w),Z_{s_1^N,\cdots, s^N_k}(z) \otimes Z_{s_1^N,\cdots, s^N_k}(w)\right)\\
= & \int_{M \times M} G_s(z,w)dV_{k}(z)dV_{k}(w)\wedge K^N_{km}(z,w)\\
=&\int_{M \times M} G_s(z,w)dV_{k}(z)dV_{k}(w)\wedge E_{km}^N(z)\wedge E_{km}^N(w)\\
&+ \int_{M \times M} G_s(z,w)dV_{k}(z)dV_{k}(w)\wedge \left[ K^N_{km}(z,w)-E_{km}^N(z)\wedge E_{km}^N(w)\right]\\
=:& I+II\\
\end{aligned}\]
where $dV_k=\frac{\omega^{m-k}}{(m-k)!}$.

Following the step in Lemma \ref{k1}, $E^N_{km}(z)$ which is the expectation of zero locus $Z_{s_1^N,\cdots, s^N_k}$ admits the following expansions,  \begin{equation}\label{ddddd2}\begin{aligned}
E^N_{km}(z)= &\left(  \frac\omega \pi N+ \frac{i\partial\bar\partial
S(z)}{\pi}N^{-1}+\cdots\right)^k\\ =&N^k\left(\frac{\omega}{\pi}\right)^k+kN^{k-2}\left(\frac{\omega}{\pi}\right)^{k-1}\wedge \frac{ i\partial\bar\partial S(z)}{\pi}+\cdots\end{aligned}
\end{equation}
Thus, it's easy to get,
$$I=\left(\frac{\pi^{m-k}}{(m-k)!}\right)^2[a_1(h,s)N^{2k}+w_2(k)a_2(h,s)N^{2k-2}+w_3(k)a_3(h,s)N^{2k-3}+\cdots]$$
where all $a_j$ are the same as in Lemma \ref{k1} and $w_j$ is the weight of $a_j$, in particular, $w_2=\frac{k}{m}$.

To the second term $II$, we need to decompose it to be,
$$\begin{aligned}II&=\int_M\int_{r_g(z,w)\leq \frac{b\log \sqrt{N}}{\sqrt N}}\cdots+\int_M\int_{r_g(z,w) \geq \frac{b\log \sqrt{N}}{\sqrt N}}\cdots\\
&=: III+IV
\end{aligned}$$
Recall the estimate (\ref{of2f}),
$$K_{km}^N(z,w)-E^N_{km}(z)\wedge E^N_{km}(w)=O(N^{-q}) \,\,\,\mbox{on} \,\,\,r_g(z,w) \geq \frac{b \sqrt{\log N}}{\sqrt{N}} $$
this is true for any $k$. Thus following the argument in Lemma \ref{toe},
$$IV =O(N^{-q+\frac{s}{2}})$$
for $q$ large enough. Thus $IV$ can be ignored in the final expansion of $E_{\mu_{h}^N}\mathcal E_{s}$.
Thus we will finish the proof if we can get the estimate of $III$.

We change variables,
$$w=z+\frac u{\sqrt N}$$
Then by the rescaling property of (\ref{current}), we have,
\begin{align}\label{become H}
K^N_{km}(z,z+\frac u{\sqrt N})=N^k\kappa_{km}(|u|)\left(\frac{\omega}{\pi}\right)^k\wedge(\frac
i{2\pi}\partial\bar\partial |u|^2)^k+O(N^{k-\frac{1}{2}})
\end{align}

Following the step in proving the estimate (\ref{K-K}), we have,
\begin{equation}{\label{aaag}}
E^N_{km}(z)\wedge E^N_{km}(z+\frac u{\sqrt N})=N ^k \left(\frac{\omega}{\pi}\right)^k\wedge
\left(\frac i {2\pi}\partial\bar\partial |u|^2 \right)^k+ O(N^{k-\frac{1}{2}})
\end{equation}


Thus we further rewrite $III$ as,
\begin{align*}III=N^{2k-m}\left(\frac{\pi^{m-k}}{(m-k)!}\right)^2\int_M\int_{|u|_g \leq b \sqrt{\log N}}
&G_s(z,z+\frac u{\sqrt N})\left[\kappa_{km}(|u|)-1\right]\\
&\cdot \left(\frac{\omega}{\pi}\right)^m\wedge \left(\frac{i}{2\pi}\partial\bar\partial |u|^2\right)^m + R_s
\end{align*}
where \begin{align*}
R_s=O(N^{2k-m-\frac{1}{2}}) \int_M\int_{|u|_g \leq b \sqrt{\log N}}
G_s(z,z+\frac u{\sqrt N})\left(\frac{\omega}{\pi}\right)^m\wedge \left(\frac{i}{2\pi}\partial\bar\partial |u|^2\right)^m
\end{align*}
Thus by the same argument in Lemma \ref{relmmma2}, we have,
\begin{equation}R_s=O\left(N^{2k-m+\frac{s-1}{2}}(\log N)^{m-\frac{s}{2}}\right) \,\,\,\text{if $0<s<2m$,}
\end{equation}
Applying the asymptotics of geodesic distance function (\ref{sidnc}) near the origin, we rewrite,
\begin{equation}\label{eq:newejffffjjjq}\begin{aligned}III= &N^{2k-m+\frac{s}{2}}\left(\frac{\pi^{m-k}}{(m-k)!}\right)^2\int_M\int_{0\leq|u| \leq \infty} \frac{1}{|u|^s}\left[\kappa_{km}(|u|)-1\right]\left(\frac{\om}{\pi}\right)^m\wedge \left(\frac{i}{2\pi}\partial\bar\partial |u|^2\right)^m\\
&+R_s+\mbox{lower order terms}\\
=&  N^{2k-m+\frac{s}{2}}\left(\frac{\pi^{m-k}}{(m-k)!}\right)^2\frac{V_{2m-1}}{\pi^m}\int_{0\leq|u| \leq \infty}\frac{1}{|u|^s}\left[\kappa_{km}(|u|)-1\right]|u|^{2m-1}d|u|\\
&+R_s+\mbox{lower order terms}\\
=&: d_m(s,k) N^{2k-m+\frac{s}{2}}+R_s+\mbox{lower order terms}
\end{aligned}\end{equation}
by polar coordinate.

The integrand in $d(s,k)$ is always integrable on the tail $a<|u|<\infty$ as $a$ large enough because of
asymptotic property $\kappa_{km}(r)=1+O(r^4e^{-r^2})$
as $r\to \infty$. To ensure $d_m(s,k)$ to be well defined near the origin, recall that when $k\neq m$,
$$\kappa_{km}(r)\sim r^{-2k} \,\,\,\mbox{as}\,\,\, r\to 0,$$ hence,  $s<2(m-k)$ is a necessary condition.

\end{proof}

Now we can turn to prove Theorem \ref{con},
\begin{proof} For the case $s\neq 0$, $s<2(m-k)$,
recall the decomposition of $E_{\mu_{h}^N}\mathcal E_{s}$,
\begin{align*}E_{\mu_{h}^N}\mathcal E_{s}=&I+II\\
=&\left(\frac{\pi^{m-k}}{(m-k)!}\right)^2[a_1(h,s)N^{2k}+w_2(k)a_2(h,s)N^{2k-2}+\cdots]+O(N^{2k-p})\\
&+d_m(s,k)N^{2k-m+\frac{s}{2}}+O\left(N^{2k-m+\frac{s-1}{2}}(\log N)^{m-\frac{s}{2}}\right)+ O(N^{-q+\frac{s}{2}})
  \end{align*} The smallest integer $p$ such that $2k-(p-1)>2k-m+\frac{s}{2}$ will ensure the above expansion is well defined, i.e, $p=[m-\frac{s}{2}]+1$ which completes the proof.  If we further choose $q$ large enough, we obtain
  \begin{align*}E_{\mu_{h}^N}\mathcal E_{s}
=&\left(\frac{\pi^{m-k}}{(m-k)!}\right)^2[a_1(h,s)N^{2k}+w_2(k)a_2(h,s)N^{2k-2}+\cdots]\\
&+d_m(s,k)N^{2k-m+\frac{s}{2}}+O\left(N^{2k-m+\frac{s-1}{2}}(\log N)^{m-\frac{s}{2}}\right)
  \end{align*}  \end{proof}
The same method is also applied to logarithm energies when $s=0$, but we omit the proof here.

\end{document}